\documentclass{amsart}

\usepackage{graphicx} 
\usepackage{amssymb}
\usepackage{amsthm}
\usepackage{tikz}
\usepackage{amsfonts}
\usepackage{xcolor}

\newtheorem{definition}{Definition}
\newtheorem*{question}{Question}
\newtheorem{theorem}{Theorem}
\newtheorem{Remark}{Remark}
\newtheorem{lemma}{Lemma}
\newtheorem{corollary}{Corollary}

\newtheorem{remark}{Remark}

\begin{document}

\author[J.\ Atchley]{James Atchley}
\address{James Atchley, Midland College, Department of Mathematics, 
3600 North Garfield, Midland TX 79705}
\email{jatchley@midland.edu}

\author[L.\ Fishman]{Lior Fishman}
\address{Lior Fishman, University of North Texas, Department of Mathematics, 
1155 Union Circle \#311430, Denton, TX 76203-5017, USA}
\email{lior.fishman@unt.edu}

\author[S.\ Jackson]{Stephen Jackson}
\address{Stephen Jackson, University of North Texas, Department of Mathematics, 
1155 Union Circle \#311430, Denton, TX 76203-5017, USA}
\email{stephen.jackson@unt.edu}

\author[D. Liu]{Daozheng Liu}
\address{Daozheng Liu, University of North Texas, Department of Mathematics, 
1155 Union Circle \#311430, Denton, TX 76203-5017, USA}
\email{daozheng.liu@unt.edu}

\author[E. Yao]{Emily Yao}
\address{Emily Yao, Princeton University, Washington Rd., Princeton, NJ 08544, USA}
\email{ey7643@princeton.edu}

\keywords{Vitali sets, Schmidt games, determinacy}
\subjclass{03E05, 03E15}

\thanks{The third author acknowledges support from NSF grant DMS 1800323.}

\title{Schmidt's game and Vitali sets}

\begin{abstract}
 While many types of non-measurable sets are never (\(\alpha, \beta\))-winning in the sense of Schmidt's game, we show that this is not the case for certain Vitali sets.  Our main theorems show that for certain values of $\alpha, \beta$ one can construct a Vitali set which is $(\alpha, \beta)$-winning, while for other values of $\alpha,\beta$ every Vitali set is $(\alpha,\beta)$-losing.  We also investigate the $(\alpha,\beta)$-Schmidt game for various other types of pathological sets, highlighting their differences from Vitali sets.
\end{abstract}

\maketitle

\section{Introduction} Recall that a Vitali set is a subset of $\mathbb{R}$ that contains exactly one representative for each element of $\mathbb{R} / \mathbb{Q}$ and is a classic example of a set that is not Lebesgue-measurable. This set, along with other sets such as Bernstein sets, Sierpinski sets, Luzin sets, and other pathological sets, is constructed using the Axiom of Choice. Recently, there has been interest in discovering the properties of these sets, as well as the relations between them. For example, Theorem 1.3 in \cite{BS} shows (in ZF set theory) that the existence of a Bernstein set does not imply the existence of a Vitali set. So in some sense, a Vitali set is even ``more pathological'' than a Bernstein set or a Luzin set.

In 1966, W. M. Schmidt \cite{Schmidt} introduced a two-player game, later referred to as Schmidt's game. This
game, as well as the well-known Banach-Mazur game (see \cite{Oxtoby}), is an intersection game. 
Schmidt invented the game
primarily as a tool for studying certain sets which arise in number theory and Diophantine approximation
theory. These sets are often null and meager, that is, exceptional with respect to both measure and category.

Schmidt's $(\alpha,\beta)$ game, with a preassigned target set $T$, is a 2-player game proceeding as follows: Given two real numbers $\alpha$ and $\beta$ strictly between $0$ and $1$ and a set $T \subset \mathbb{R}$, player one (named Bob) makes the initial move, which is any closed interval $B_0$ of any positive diameter $\rho$. Next, player two (named Alice) responds with a closed interval $A_0 \subset B_0$ with a diameter of $(\alpha) \rho$. Bob then follows up with his move $B_1 \subset A_0$ of diameter $(\alpha\beta) \rho$. Alice then responds with $A_1 \subset B_1$ of diameter $\alpha (\alpha\beta) \rho$, and so on. Alice wins if and only if  \(\{x\} = \bigcap_{n=0}^{\infty}A_n \subseteq T\).

A {\em strategy} for one of the players in the game is a function which, given a finite sequence of moves for the other player, will produce a move for the player and is such that the move given by the strategy will follow the rules of the game provided the previous moves do (the rules for Schmidt's game are the diameter and nesting requirements on the moves).  A {\em winning strategy} for one of the players is a strategy for that player such that any run of the game in which the player follows the strategy at every move results in a win for that player. 

\begin{definition}
    $T \subseteq \mathbb{R}$ is said to be $(\alpha,\beta)$-winning if Alice has a winning strategy when $T$ is the target set. If, on the other hand, Bob has a winning strategy, then $T$ is said to be $(\alpha,\beta)$-losing. In either of these cases, $T$ is said to be determined.  Otherwise, it is said to be undetermined or more precisely \((\alpha, \beta)\) undetermined.
\end{definition}

\begin{definition}
    A set of real numbers $T$ is said to be $\alpha$-winning for a fixed $\alpha$ if it is $(\alpha, \beta)$-winning for all $\beta$.
\end{definition}

\begin{remark}
    There are some choices for \((\alpha, \beta)\) where every dense target set is \(\alpha, \beta\) winning (in particular, when $\alpha < 2 - \frac{1}{\beta}$). On the other hand, there are choices where only the entire space is \((\alpha, \beta)\)-winning (in particular, when $\beta < 2 -\frac{1}{\alpha}$). We will refer to these as trivial cases and exclude them in some theorems.  To see these situations described more clearly, see \cite{AFS}.
\end{remark}

One of the properties established by \cite{Schmidt} is that the property of $\alpha$-winning sets is closed under countable intersections. Thus, the collection of $\alpha$-winning sets forms a countably additive filter, similar to the co-meager sets in the Banach-Mazur game.

In this paper our main result is to show that even though there is no Vitali set that is \(\alpha\)-winning (Theorem \ref{awin} and Theorem \ref{abwin}) there \emph{does} exist an \(\alpha, \beta \)-winning Vitali set for some values of \(\alpha\) and \(\beta\) (Theorem \ref{main}). Note that not all Vitali sets are determined.  It follows from a result of Michalski \cite{M} that there is a Vitali set for which Schmidt's game is not determined (see section \ref{NotVitali}).  The final part of this paper is dedicated to Schmidt's game determinacy results for other pathological sets such as Luzin, Sierpinski, and Bernstein sets.

\section{Vitali Sets are Not $\alpha$-winning}
In this section, we prove that for any Vitali set and any value of $\alpha > 0$, there is a $\beta$ such that the Vitali set is not $(\alpha,\beta)$-winning. In particular, we show that if $\beta < \alpha$, a Vitali set is not $(\alpha,\beta)$-winning.

We first give a simple non-constructive proof that no Vitali set is $\alpha$-winning.
\begin{theorem}\label{awin} No Vitali Set is $\alpha$-winning.
\end{theorem}
\begin{proof} Let $V$ be a Vitali Set, and $V_q = V + q$ with $q \in \mathbb{Q} \setminus \{0\}$. Suppose by way of contradiction that $V$ is $\alpha$-winning. It is clear that $V_q$ must also be $\alpha$-winning since we can transfer Bob's moves on $V_q$ to $V$ and Alice's following winning moves on $V$ to $V_q$. Then Alice wins on $V_q$, which means that $V\cap V_q$ is nonempty since it is $\alpha$-winning. However, for all $q \in \mathbb{Q} \setminus \{0\}$, $V_q \cap V = \emptyset$, which is a contradiction to the countable intersection property of $\alpha$-winning sets. 
\end{proof}

We present now a more constructive proof that for $\beta < \alpha$ no Vitali set is $(\alpha,\beta)$-winning which motivates our later results. 
\begin{theorem}\label{abwin} If $ \beta < \alpha $ then no Vitali set is $(\alpha,\beta)$-winning.
\end{theorem}
\begin{proof} Suppose, by way of contradiction, that Alice does have a winning strategy which we will call $\tau$. We will show that there is always a way for Bob to ``foil" Alice's strategy.

Fix a Vitali set, $V$. Let Bob and Alice play two games on $V$.
To begin the two games, Bob chooses two disjoint intervals $B_0$ and $B_0^\prime$ where $B_0$ is the first move for the first game and $B_0^\prime$ is the first move for the second game such that $$\beta|B_0|\leq|B_0^\prime|\leq\alpha|B_0|$$ \\[-1 ex]
and $|B_0| = \rho$, $|B_0^\prime| = \rho^\prime$ denote the lengths of $B_0$ and $B_0^\prime$ respectively.

Alice then takes her turn, using her winning strategy $\tau$ to play $A_0$ and $A_0^\prime$ on each of the games. By the above inequality, there exists $q \in \mathbb{Q} \setminus \{0\}$ such that $A_0^\prime = \tau(B_0^\prime) \subseteq A_0 + q = \tau(B_0) + q$. Since $\rho^\prime\leq\alpha\rho$, $|A_0^\prime| = \alpha\rho^\prime\leq\alpha^2\rho<\alpha\rho$.

Now it is Bob's turn again. Let him first play $B_1$. We can see that $|B_1| = \alpha\beta\rho < \alpha\rho^\prime=|A_0^\prime|$. Thus, Bob may choose $B_1$ so that $B_1 + q \subseteq A_0^\prime$. Alice then plays $A_1 = \tau(B_0,B_1).$ Notice that $A_1 + q \subseteq A_0^\prime$. Bob can then choose $B_1^\prime \subseteq A_1 + q$ since $|B_1^\prime| = \alpha\beta\rho^\prime \leq \alpha^2\beta\rho = |A_1| = |A_1 + q|$. Then $A_1^\prime \subseteq A_1 + q$. 
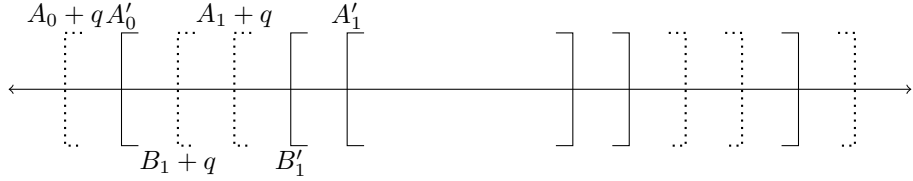
\begin{center}
\begin{tikzpicture}[scale=0.75]
    \draw[<->] (-8,0) -- (8,0);

    \node at (-7, 1.3) {$A_0+q$};
    \draw[thick, dotted] (-6.7, 1) -- (-7, 1) -- (-7, -1) -- (-6.7, -1);
    \draw[thick, dotted] (6.7, 1) -- (7, 1) -- (7, -1) -- (6.7, -1);

    \node at (-6, 1.3) {$A_0^\prime$};
    \draw (-5.7, 1) -- (-6, 1) -- (-6, -1) -- (-5.7, -1);
    \draw (5.7, 1) -- (6, 1) -- (6, -1) -- (5.7, -1);

    \node at (-5, -1.3) {$B_1+q$};
    \draw[thick, dotted] (-4.7, 1) -- (-5, 1) -- (-5, -1) -- (-4.7, -1);
    \draw[thick, dotted] (4.7, 1) -- (5, 1) -- (5, -1) -- (4.7, -1);

    \node at (-4, 1.3) {$A_1+q$};
    \draw[thick, dotted] (-3.7, 1) -- (-4, 1) -- (-4, -1) -- (-3.7, -1);
    \draw[thick, dotted] (3.7, 1) -- (4, 1) -- (4, -1) -- (3.7, -1);

    \node at (-3, -1.3) {$B_1^\prime$};
    \draw (-2.7, 1) -- (-3, 1) -- (-3, -1) -- (-2.7, -1);
    \draw (2.7, 1) -- (3, 1) -- (3, -1) -- (2.7, -1);

    \node at (-2, 1.3) {$A_1^\prime$};
    \draw (-1.7, 1) -- (-2, 1) -- (-2, -1) -- (-1.7, -1);
    \draw (1.7, 1) -- (2, 1) -- (2, -1) -- (1.7, -1);
    
\end{tikzpicture} \end{center}
In general, given that $A_n^\prime   \subseteq A_n+q$, it follows by the same argument that Bob can pick $B_{n+1}$ and $B_{n+1}^\prime$ such that 
$A_{n+1}^\prime \subseteq A_{n+1} + q$. Thus, $\bigcap\limits_{n = \mathbb{N}}A_n^\prime = \bigcap\limits_{n = \mathbb{N}}A_n + q$ By the nested interval property, the first game ends in a point $\{x\} = \bigcap\limits_{n = \mathbb{N}}A_n$ and the second game ends in a point $\{x^\prime\} = \bigcap\limits_{n = \mathbb{N}}A_n^\prime$. Then $\{x^\prime\} = \bigcap\limits_{n = \mathbb{N}}A_n^\prime = \bigcap\limits_{n = \mathbb{N}}A_n + q = \{x + q\}$.

Since $\tau$ is a winning strategy, both $x$ and $x^\prime$ must be in the same Vitali set $V$. However, they differ by some $q \in \mathbb{Q} \setminus \{0\}$ and thus must be in the same equivalence class, resulting in a contradiction.
\end{proof} 
\section{$(\alpha,\beta)$-winning Vitali sets}
In this section, we show that for any $\beta > 0$ if $\alpha$ is small enough, then there is a Vitali set that is $(\alpha,\beta)$-winning. Our proof is constructive and we produce a partial Vitali set that is $(\alpha,\beta)$-winning without utilizing the Axiom of Choice. We do this in stages. We first show that for any fixed initial move $I$, there exists a Vitali set that is $(\alpha,\beta, I)$-winning (that is Alice has a winning strategy when Bob plays $I$ as their first move). Here we only need $\alpha < \min\{\frac{1}{12}, \beta\}$. We then extend the argument to show (Theorem 3) that for any fixed $\rho$, there is a Vitali set that is $(\alpha,\beta, \rho)$-winning (That is Alice has a winning strategy when Bob's first move is any ball of diameter $\rho > 0$). Finally we extend the argument, further requiring that $\beta > \frac{\alpha}{(1-5\alpha)^2}$, to show that there is a Vitali set that is $(\alpha,\beta)$-winning.

\begin{definition}\label{partial} A partial Vitali set is a set that contains at most one representative for each element of $\mathbb{R} / \mathbb{Q}$.
\end{definition}
\begin{Remark}
    
We will construct a partial Vitali set that is $(\alpha, \beta)$ winning. As mentioned above, this does not require the Axiom of Choice.  Extending this set to a full Vitali set is our only use of AC.

\end{Remark}

\begin{lemma}\label{abi} If $\alpha$ is sufficiently small $(\alpha < \frac{1}{12})$ and $\alpha < \beta$,  then for any fixed first move $I$ there exists a Vitali set $V_I$ such that $V_I$ is $(\alpha, \beta, I)$-winning.
\end{lemma}
\begin{proof} Enumerate $\mathbb{Q}$ and consider $q_0$. Consider Bob's first move $B_0$ and let $|B_0| = \rho$. Then let Alice play her move $A_0$ in the middle (or close to the middle) of $B_0$. Let $\omega^1_1, \cdots, \omega^1_m$ be a collection of intervals of length $4\alpha(\alpha\beta)\rho$ that are separated by more than $4\alpha(\alpha\beta)\rho$ such that any of Bob's potential moves must contain at least one interval. If $A_0$ is the interval from $0$ to $\beta\rho$, place left endpoint of $\omega^1_l$ at $l\alpha\beta\rho(1-4\alpha)$ the right endpoint at $l\alpha\beta\rho(1-4\alpha) + 4\alpha(\alpha\beta)\rho$. Notice that if $\alpha<\frac{1}{12}$, then the distance between $\omega_l$ and $\omega_{l+1}$ is at least the length of $\omega_l$.

If $4\alpha(\alpha\beta)\rho > q_0$, take the middle of each $\omega^1$ as potential Alice moves $A^1$'s and construct $\omega^2_1, \cdots, \omega^2_{m^2}$ by the same process as the $\omega^1$'s for each $A^1$. Notice that the size of these intervals is $4\alpha(\alpha\beta)^2\rho$. Continue this process, taking the middle of each $\omega^n$ as $A^n$ until $4\alpha(\alpha\beta)^n\rho < q_0$. Let $n_0$ be the least such $n$.

We claim that the shift of any $\omega^n$ by $q_0$ will intersect at most one other $\omega^n$. Consider some $\omega^n_1 + q_0$. Suppose, for the sake of contradiction, that there exist $m_1, m_2$ such that $\omega^n_{m_1} \cap \omega^n_1 + q_0 \neq \emptyset$ and $\omega^n_{m_2} \cap \omega^n_1 + q_0 \neq \emptyset$ . Then $|\omega^n_1| > \big((\alpha\beta)^n - 8\alpha(\alpha\beta)^n)\rho > 4\alpha(\alpha\beta)^n\rho$ (the length of the gap between intervals by construction), which contradicts $|\omega^n_1|=4\alpha(\alpha\beta)^n\rho$. Thus, $\omega^n_1 + q_0$ intersects at most one other $\omega^n$. Similarly, for any $\omega^n_i$, there is at most one other $\omega^n_m$ such that if we shift $\omega^n_m$, we will intersect $\omega^n_i$.

We construct a series of chains, where each chain is a series of $\omega^n$'s that intersect the next $\omega^n$ when they are shifted by $q_0$. Consider the equivalence relation generated by the relation $\omega^n_i\sim \omega^n_j $ if $\omega^n_i + q_0$ intersects $\omega^n_j$. Let each chain be an equivalence class. By the previous paragraph, each equivalence class is a chain of these $\omega_n$'s where each interval shifted by $q_0$ intersects the next one.

For each $\omega^n_i$, we now define Alice's potential move, $A^n_i \subset \omega^n_i$. Consider a chain, and let $\omega^n_{i_0}$ be the first element of its chain. Let $A^n_{i_0} \subset \omega^n_{i_0}$ be an arbitrary interval of the right size. Let the next element of the chain be $\omega^n_{i_1}$ and let $A^n_{i_1} \subset \omega^n_{i_1}$ be such that $A^n_{i_0} + q_0$ is disjoint from $A^n_{i_1}$ (we can do this since $|\omega^n_{i_1}| > 3|A^n_{i_1}|$). We continue in a similar manner to define $A^n_i$ for every $\omega^n_i$ in the chain and repeat this process for all chains.

Now consider our second rational, $q_1$. Have the construction of choosing plays in the middle of each $\omega^n$ be the same as above until we reach an $n_1$ where $4\alpha(\alpha\beta)^{n_1}<q_1$ $(n_1 > n_0)$ and repeat the above construction for this $n_1$ to define $A^{n_1}_{i}$.

Eventually, we will do this for every rational number $q_j$, that is, we define $A^{n_j}_i$ for all $j$. This tree of resulting $A^{n_j}_i$'s gives a perfect set $P_I$ that is also a partial Vitali set. Pick an element from each missing Vitali equivalence class to get a full Vitali set, and we have constructed a Vitali set that is $(\alpha, \beta, I)$-winning.
\end{proof}

\begin{lemma}\label{abi2} If $\alpha < \frac{1}{12}$, $\alpha<\beta$, and $\rho > 0$, and $I_1$ and $I_2$ are two different intervals of length $\rho$, then there exists a Vitali set that is both $(\alpha, \beta, I_1)$-winning and $(\alpha, \beta, I_2)$-winning.
\end{lemma}

\begin{proof} Fix the intervals $I_1$ and $I_2$, both of length $\rho$. We will define the Alice moves $A^{n_j}_i (I_1)$ and $A^{n_j}_i (I_2)$ for the two intervals respectively. We will define $P_{I_1}$ and $P_{I_2}$ as before, but we will also ensure that $P_{I_1} \cup P_{I_2}$ is a partial Vitali set. 

At even stages $2j$ of the construction, we repeat the construction of Lemma 1 for the rational $q_j$. This will ensure that each $P_{I_1}$ and $P_{I_2}$ is a partial Vitali set. At odd stages $2j + 1$, we will ensure that each $A^{n_j}_i (I_1) + q_j$ is disjoint from each $A^{n_j}_{i^\prime} (I_2)$. If we do this, $P_{I_1} \cup P_{I_2}$ will be a partial Vitali set. Go to the next level of the construction for both $I_1$ and $I_2$, $\omega^{n_{2j + 1}}$. Consider the intervals $\omega^{n_{2j + 1}}_i (I_1)$ and $\omega^{n_{2j + 1}}_{i^\prime} (I_2)$ for the two sides of the construction. Note that by the same argument as Lemma 1, every $\omega^{n_{2j + 1}}_i (I_1) + q_j$ intersects at most one $\omega^{n_{2j + 1}}_{i^\prime} (I_2)$. For every such pair of $i, i^\prime$, choose $A^{n_{2j + 1}}_i (I_1) \subset \omega^{n_{2j + 1}}_i (I_1)$ and $A^{n_{2j + 1}}_{i^\prime} (I_2) \subset \omega^{n_{2j + 1}}_{i^\prime} (I_2)$ such that $A^{n_{2j + 1}}_i (I_1) + q_j$ is disjoint from $A^{n_{2j + 1}}_{i^\prime} (I_2)$. Then for any $x \in P_{I_1}$ and $y \in P_{I_2}$, $x + q_j \neq y$. Since $-q_j$ is also considered at some stage, $y + q_j \neq x$. Thus, $|x - y| \neq q_j$, and thus $P_{I_1} \cup P_{I_2}$ is a partial Vitali set.
\end{proof}

\begin{corollary}\label{countably} If $\alpha < \frac{1}{12}$, $\alpha<\beta$, and $\rho > 0$, and there are countably many intervals $I_1, I_2, \cdots$, all of length $\rho$, then there exists a Vitali set that is $(\alpha, \beta, I_n)$-winning for all $n$.
\end{corollary}

\begin{proof}
Enumerate all triples $q_j, k_j, l_j$ where $q_j$ is a rational and $k_j$ and $l_j$ are integers. At step $j$ of the construction, repeat Lemma \ref{abi2} for $I_{k_j}$ and $I_{l_j}$ for $q_j$. Then the union of the $P_{I_n}$ is a partial Vitali set. 
\end{proof}
\noindent We modify slightly the construction of Lemma \ref{abi2} to get the following lemma. 
\begin{lemma}\label{epsclose} Fix $\alpha$, $\beta$, $I_1,$ and $I_2$ where $I_1$ and $I_2$ are intervals of length $\rho$. Then there is a Vitali set and an $\epsilon > 0$ such that if $I$ is an interval of length $\rho$ whose endpoints are $\epsilon$ close to $I_1$ or $I_2$ then the Vitali set is $(\alpha,\beta,I)$-winning. 
\end{lemma}
\begin{proof} Repeat the construction of Lemma \ref{abi2} but ensure that the $\omega$ intervals are valid (are contained in any Bob move) for each $I$ that is $\epsilon$ close, i.e. whose endpoints are at most $\epsilon$ away from the endpoints of $I_1$ or $I_2$. Now, for the first level of the construction, we have $\omega$ intervals of length $(4-2\epsilon)(\alpha(\alpha\beta)\rho)$ that are contained by an initial Bob move in any interval $I$ that is $\epsilon$ close. Note that after the initial move, since we can use the same Alice moves, the $\omega$ intervals can be constructed just as before. Thus the partial Vitali set $P_{I_1} \cup P_{I_2}$ is $(\alpha,\beta,I)$-winning for any $I$ that is $\epsilon$ close to $I_1$ or $I_2$. \end{proof}
\begin{corollary}\label{epsclose2} If $\alpha < \frac{1}{12}, \alpha < \beta,$ and $\rho > 0$ and we fix countably many intervals $I_1, I_2, \cdots$ all of length $\rho$, then there is a Vitali set and an $\epsilon > 0$ such that if $I$ is an interval of length $\rho$ whose endpoints are $\epsilon$ close to one of $I_1, I_2, \cdots$ then the Vitali set is $(\alpha,\beta,I)$-winning.
\end{corollary}
\begin{proof} Repeat the argument of Lemma \ref{epsclose} but diagonalize over the $I_n$'s as in Corollary \ref{countably}.
\end{proof}
\begin{theorem}\label{abrho} If $\alpha < \frac{1}{12}$, $\alpha < \beta$, and $\rho > 0$ there exists a Vitali set $V$ which is $(\alpha,\beta,\rho)$-winning.
\end{theorem}
\begin{proof} Let $I_1, I_2, \cdots$ be the intervals with a rational left endpoint of length $\rho$. Let $\epsilon$ and a Vitali set $V$ be as in Corollary \ref{epsclose2}. Then every interval $I$ of length $\rho$ is $\epsilon$ close to one of the $I_n$'s. So by Corollary \ref{epsclose2}, $V$ is $(\alpha,\beta,\rho)$-winning.
\end{proof}
\begin{lemma}\label{i1i2} Let $\alpha < \frac{1}{12}, \beta > \frac{\alpha}{(1-8\alpha)^2}$, and let $I_1, I_2$ be given intervals of length $\rho_1$ and $\rho_2$ respectively. Then there is a Vitali set $V$ that is $(\alpha, \beta, I_1)$-winning and $(\alpha, \beta, I_2)$-winning.
\end{lemma}

\begin{proof} After round $m$ of the $I_1$ game, the length of Alice's interval is $\alpha(\alpha\beta)^m\rho_1$, and likewise after round $n$ of the $I_2$ game, the length of Alice's interval is $\alpha(\alpha\beta)^n\rho_2$. Choose m and n such that $\sqrt{\alpha\beta} \leq \frac{\rho_2}{\rho_1}(\alpha\beta)^{n-m} \leq \frac{1}{\sqrt{\alpha\beta}}$. Enumerate the rationals, $q_1, q_2, \cdots \in \mathbb{Q}$ as before. Without loss of generality, the lengths of the moves $\alpha(\alpha\beta)^m\rho_1$ and $\alpha(\alpha\beta)^n\rho_2$ are both less than $q_j$, some rational. At even rounds $2j$, we will play exactly as in the proof of Lemma \ref{abi2}. For odd rounds $2j+1$, we will now modify the argument of Lemma \ref{abi2}. First choose $m_{2j+1}$ and $n_{2j+1}$ such that $\sqrt{\alpha\beta} \leq \frac{\rho_2}{\rho_1}(\alpha\beta)^{n_{2j+1}-m_{2j+1}} \leq \frac{1}{\sqrt{\alpha\beta}}$. Without loss of generality, assume that $\alpha(\alpha\beta)^{n_{2j+1}}\rho_2<\alpha(\alpha\beta)^{m_{2j+1}}\rho_1$. Define the $\omega^{m_{2j+1}}_i$ and $\omega^{n_{2j+1}}_i$ intervals just as before. Consider one of the $\omega^{m_{2j+1}}_i$'s. Just as before, we have $|\omega^{m_{2j+1}}_i| = 4\alpha(\alpha\beta)^{m_{2j+1}}\rho_1$. On the right game, let Alice move in the middle of each $\omega^{n_{2j+1}}_i$. The gap between any 2 moves by Alice is at least $\rho_2\alpha(\alpha\beta)^{n_{2j+1}}(\beta-8\alpha\beta)$. Alice's move on the left game must be of size $\alpha(\alpha\beta)^{m_{2j+1}}\rho_1$. We require that $\alpha(\alpha\beta)^{m_{2j+1}}\rho_1 < \rho_2\alpha(\alpha\beta)^{n_{2j+1}}(\beta-8\alpha\beta)$, which is true for all $\beta > \frac{\alpha}{(1-8\alpha)^2}$. This guarantees that the gap between Alice's right moves is greater than the size of Alice's left moves. Note also that the the sum of the lengths of two of Alice's left moves and one of Alice's right moves is still less than the size of $\omega^{m_{2j+1}}_i$, and thus we can pick Alice's left moves such that they are disjoint from $A^{n_{2j+1}}_{i^\prime} - q_j$ for all $i$. This defines Alice's moves $A^{m_{2j+1}}_i$. Thus, we have ensured that $A^{m_{2j+1}}_i + q_j$ is disjoint from $A^{n_{2j+1}}_{i^\prime}$ for all $i^\prime$. This constructs a partial Vitali set which is $(\alpha, \beta, I_1)$-winning and $(\alpha, \beta, I_2)$-winning.
\end{proof}

\begin{definition}\label{thick} Given an interval $I = [a,b]$, a $\delta$-thickening of $I$ is the interval $[a - \delta(b-a), b + \delta(b-a)]$.
\end{definition}

\begin{lemma}\label{betaeps} If $\alpha < \frac{1}{12}$, $\beta > \frac{\alpha}{(1-5\alpha)^2}$, then there is a $\delta = \delta(\alpha, \beta)$ such that if $I_1, I_2$ are given, there is a partial Vitali set which is  $(\alpha, \beta, I)$-winning for any $I$ that is less than a $\delta$-thickening of $I_1$ or of $I_2$
\end{lemma}
\begin{proof} 
For $I_1$ and $I_2$, define $\omega$ intervals as before. Define $\omega^\prime$ intervals such that they have the same center as the $\omega$ intervals but with length $\frac{7}{2}\alpha(\alpha\beta)^n$ instead of $4\alpha(\alpha\beta)^n$. As in Lemma \ref{i1i2}, we can play arbitrarily in the two games until the ratio of Alice's moves in the two games are between $\sqrt{\alpha\beta}$ and $\frac{1}{\sqrt{\alpha\beta}}$. Let Alice's moves in the $I_2$ game be the smaller moves. Simplifying notation, let Alice's first moves be $A_1 (I_1)$ and $A_1 (I_2)$. So $1 \leq \frac{A_1 (I_1)}{A_1 (I_2)} \leq \frac{1}{\sqrt{\alpha\beta}}$.

Let $\delta \leq \frac{1}{4}\alpha\frac{\beta}{1-2\beta}$. We claim that, for any $I$ that is less than or equal to a $\delta$-thickening of $I_1$, each $\omega_i (I)$ interval contains the $\omega^\prime_i (I_1)$ interval. Consider the interval $\omega_1 (I)$. If we take the left endpoint of $I_1$ to be 0, the right endpoint of $\omega_1 (I)$ is $(1 + 2\delta)\alpha\beta|\alpha I_1| - \delta\alpha|\alpha I_1|$, which is greater than the right endpoint of $\omega^{\prime}_1 (I_1)$, $\alpha\beta|\alpha I_1| - \frac{1}{4}\alpha^2\beta|\alpha I_1|$ (which follows from the value of $\delta$). A similar inequality holds for the rightmost $\omega$. For all other $\omega$ intervals, the claim is immediate since $|\omega_i (I)| \geq |\omega_i(I_1)|$ for all $i$.

As before, we can pick the next Alice moves inside each $\omega^\prime$ interval. For each $\omega^\prime_i (I_2)$, center the Alice move $A_2$ in the center of $\omega^\prime_i (I_2)$. For each $\omega^\prime_i (I_1)$, we now define $A_2 (I_1)$. If $I$ is an $\delta$-thickening of $I_1$, $A_2$ of $I_1$ will have the same center as $A_2$ of $I$, and likewise for if $I$ is a $\delta$-thickening of $I_2$. Recall that the length of $\omega^\prime (I_1)$ is $\frac{7}{2}\alpha(\alpha\beta) |I_1|$. Consider $q_0 \in \mathbb{Q}$ again, and let $C_2$ be the center points of the $A_2 (I_2)$ moves. Note that the distance between any two distinct points in $C_2$ is at least $d=|I_2|\alpha\beta-4|I_2|\alpha(\alpha\beta)$. If the distance between the left endpoint of $\omega^\prime (I_1)$ and the least point of $C_2$ in the interval $\omega^\prime (I_1)$ (if any) is at least $|I_1|\alpha(\alpha\beta) \cdot (1+2\delta)+|I_2|\alpha(\alpha\beta) \cdot (1+2\delta)$, then $A_2 (I_1)$ will have its left endpoint as the left endpoint of $\omega^\prime (I_1)$. Otherwise, $A_2 (I_1)$ will be centered at $|I_1|\alpha(\alpha\beta) \cdot (1+2\delta)+|I_2|\alpha(\alpha\beta) \cdot (1+2\delta) + \frac{1}{2}|I_1|\alpha(\alpha\beta) \cdot (1+2\delta)$. 
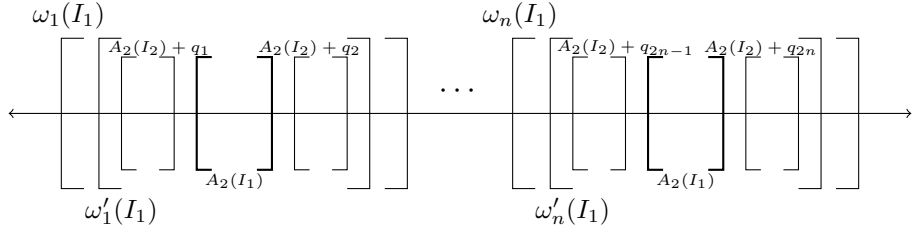
\begin{center}
\begin{figure}
\begin{tikzpicture}
    \draw[<->] (-8,0) -- (4,0);

    \node at (-7.2, 1.3) {$\omega_1(I_1)$};
    \draw (-7, 1) -- (-7.3, 1) -- (-7.3, -1) -- (-7, -1);
    \draw (-3, 1) -- (-2.7, 1) -- (-2.7, -1) -- (-3, -1);

    \node at (-6.5, -1.3) {$\omega_1^\prime(I_1)$};
    \draw (-6.5, 1) -- (-6.8, 1) -- (-6.8, -1) -- (-6.5, -1);
    \draw (-3.5, 1) -- (-3.2, 1) -- (-3.2, -1) -- (-3.5, -1);

    \node at (-6, 0.87) {\tiny{$A_2(I_2) + q_1$}};
    \draw (-6.3, 3/4) -- (-6.5, 3/4) -- (-6.5, -3/4) -- (-6.3, -3/4);
    \draw (-6, 3/4) -- (-5.8, 3/4) -- (-5.8, -3/4) -- (-6, -3/4);

    \node at (-4, 0.87) {\tiny{$A_2(I_2) + q_2$}};
    \draw (-3.7, 3/4) -- (-3.5, 3/4) -- (-3.5, -3/4) -- (-3.7, -3/4);
    \draw (-4, 3/4) -- (-4.2, 3/4) -- (-4.2, -3/4) -- (-4, -3/4);

    \node at (-5, -0.9) {\tiny{$A_2(I_1)$}};
    \draw[thick] (-4.7, 3/4) -- (-4.5, 3/4) -- (-4.5, -3/4) -- (-4.7, -3/4);
    \draw[thick] (-5.3, 3/4) -- (-5.5, 3/4) -- (-5.5, -3/4) -- (-5.3, -3/4);

    \node at (-2, 0.3) {\Large{$\cdots$}};

    \node at (-1.2, 1.3) {$\omega_n(I_1)$};
    \draw (3, 1) -- (3.3, 1) -- (3.3, -1) -- (3, -1);
    \draw (-1, 1) -- (-1.3, 1) -- (-1.3, -1) -- (-1, -1);

    \node at (-0.5, -1.3) {$\omega_n^\prime(I_1)$};
    \draw (2.5, 1) -- (2.8, 1) -- (2.8, -1) -- (2.5, -1);
    \draw (-0.5, 1) -- (-0.8, 1) -- (-0.8, -1) -- (-0.5, -1);

    \node at (0.2, 0.87) {\tiny{$A_2(I_2) + q_{2n-1}$}};
    \draw (2.3, 3/4) -- (2.5, 3/4) -- (2.5, -3/4) -- (2.3, -3/4);
    \draw (2, 3/4) -- (1.8, 3/4) -- (1.8, -3/4) -- (2, -3/4);

    \node at (2, 0.87) {\tiny{$A_2(I_2) + q_{2n}$}};
    \draw (-0.3, 3/4) -- (-0.5, 3/4) -- (-0.5, -3/4) -- (-0.3, -3/4);
    \draw (0, 3/4) -- (0.2, 3/4) -- (0.2, -3/4) -- (0, -3/4);

    \node at (1, -0.9) {\tiny{$A_2(I_1)$}};
    \draw[thick] (0.7, 3/4) -- (0.5, 3/4) -- (0.5, -3/4) -- (0.7, -3/4);
    \draw[thick] (1.3, 3/4) -- (1.5, 3/4) -- (1.5, -3/4) -- (1.3, -3/4);

\end{tikzpicture} 
\end{figure}
\end{center}
We first claim that for any possible expansion $I$ of $I_1$, the right endpoint of $A_2 (I)$ is to the left of the right endpoint of $\omega^\prime (I_1)$. That is, we must have $\frac{|I_2|}{|I_1|}(1+2\delta) \leq \frac{3}{2} - 4\delta$. Since $\frac{|I_2|}{|I_1|} \leq 1$, it suffices to have $(1+2\delta) \leq \frac{3}{2} - 4\delta$, i.e. $\delta \leq \frac{1}{12}$. In this case, $A_2 (I)$ will be a valid move for Alice, that is, it is inside the $\omega^\prime (I_1)$ interval.

Now we check that $A_2 (I)$ will not intersect $A_2 (I^\prime) + q_0$, where $I$ and $I^\prime$ are expansions of $I_1$ and $I_2$ respectively. Either this is clear or the center of $A_2 (I)$ is at the midpoint of two consecutive elements in $C_2$ of $\omega^\prime$. In this case, the interval between the two corresponding elements of $A_2 (I^\prime)$ is at least $d - |I_2|\alpha(\alpha\beta)(1+2\delta)$ and the length of $A_2 (I)$ is at most $|I_1|\alpha(\alpha\beta)(1+2\delta)$ So it suffices to check that $d - |I_2|\alpha(\alpha\beta)(1+2\delta) \geq |I_1|\alpha(\alpha\beta)(1+2\delta)$ i.e. $|I_2|\alpha\beta-4|I_2|\alpha(\alpha\beta) \geq |I_1 + I_2| \alpha (\alpha\beta)(1+2\delta)$. This simplifies to $\frac{|I_2|}{|I_1|} \geq \frac{\alpha(1+2\delta)}{1-4\alpha - \alpha(1+2\delta)}$ Since $\frac{|I_2|}{|I_1|} \geq \sqrt{\alpha\beta}$, it suffices to check that $\sqrt{\alpha\beta} \geq \frac{\alpha(1+2\delta)}{1-4\alpha - \alpha(1+2\delta)}$, which simplifies to $\beta \geq \alpha \cdot (\frac{\alpha(1+2\delta)}{1-4\alpha - \alpha(1+2\delta)})^2$. Since $\beta > \frac{\alpha}{(1-5\alpha)^2}$, we may choose $\delta$, depending only on $\alpha$ and $\beta$, small enough such that the inequality is satisfied. Thus, $A_2 (I)$ will not intersect $A_2 (I^\prime) + q_0$. This completes the definition of Alice's moves $A_2 (I_1)$ and $A_2 (I_2)$, and thus also $A_2 (I)$ and $A_2 (I^\prime)$ for any expansions $I$ and $I^\prime$. The future moves are now defined in exactly the same way starting from $A_2 (I_1)$ and $A_2 (I_2)$ instead of $A_1 (I_1)$ and $A_1 (I_2)$ using the rationals $q_i$ at the level $2i + 2$ (and at odd levels we construct $A_i$ as in Lemma \ref{abi}). Thus, we have constructed a partial Vitali set on which any given $I$ that is less than a $\delta$-thickening of $I_1$ or of $I_2$ is $(\alpha,\beta,I)$-winning.

\end{proof}
\begin{theorem}{\label{main}} If $\alpha < \frac{1}{12}$ and $\beta > \frac{\alpha}{(1-5\alpha)^2}$ then there is a Vitali set that is $(\alpha, \beta)$-winning.
\end{theorem}
\begin{proof}
Let $\delta = \delta (\alpha, \beta)$ as in Lemma \ref{betaeps}. Enumerate all the intervals $I_n$ with rational endpoints. As in the proof of Corollary \ref{epsclose2} but instead using Lemma \ref{betaeps}, we can construct a partial Vitali set which is $(\alpha, \beta, I)$-winning, where $I$ is any $\delta$-thickening of any $I_n$, and thus is $(\alpha,\beta)$-winning.
\end{proof}

\section{Bernstein, Luzin, and Sierpinski sets}{\label{NotVitali}}
We recall the definitions of the following sets.
\begin{definition}
A {\em Bernstein set} is a set $B\subseteq \mathbb{R}$ such that for every nonempty perfect set $P\subseteq \mathbb{R}$ we have $B\cap P\neq \emptyset$ and
$B^c \cap P \neq \emptyset$. A {\em Luzin} set 
$L\subseteq \mathbb{R}$ is a set such that for every for every meager set 
$A\subseteq \mathbb{R}$ we have $L\cap A$ is countable. A {\em Sierpinski set} 
$S\subseteq \mathbb{R}$ is a set such that for every measure $0$ set $A\subseteq \mathbb{R}$ we have $S\cap A$ is countable.

\end{definition}

We begin by making several observations on the winning property of these other pathological sets.
\begin{theorem}\label{bernstein}In any non-trivial case, Bernstein sets are not $(\alpha, \beta)$-determined.
\end{theorem}
This is shown in \cite{AFS}.  The idea of the proof is showing that, in any non-trivial case, an $(\alpha, \beta)$-determined set must either contain a perfect set or it's complement must, so a Bernstein set cannot be determined.

\begin{theorem} For every non-trivial value of $\alpha$ and $\beta$, Alice does not have a winning strategy on a Luzin or Sierpinski set.
\end{theorem}
\begin{proof} From a winning strategy for Alice, as in Theorem 3 from \cite{AFS}, we construct a perfect set contained in the target set. It is easy to see that this perfect set is both measure zero and nowhere dense. 
\end{proof}
\begin{theorem} If $\beta < \frac{1}{3}$, then for any Luzin or Sierpinski set, Bob has a winning strategy in the $(\alpha, \beta)$ Schmidt game. 
\end{theorem}
\begin{proof} Let Bob's first move be the unit interval, and let $S$ be a Sierpinski set. Construct disjoint intervals of length $\alpha\beta$ in the interval. These intervals will be separated by $\frac{\alpha - 3\alpha\beta}{3}$, which is possible when the first interval shares its left endpoint with the left endpoint of Bob's first move. We continue in this manner to build a Cantor-like set $P$, specifically at each stage $n$ we have a finite set of disjoint intervals of size $(\alpha\beta)^n$. Inside each of these intervals put disjoint intervals of length $(\alpha\beta)^{n+1}$ separated by $(\alpha\beta)^{n}(\frac{\alpha - 3\alpha\beta}{3})$. We can clearly see that $P$ has measure zero. Since $S$ is a Sierpinski set, $S \cap P = \{a_0, a_1, \cdots\}$ is countable.

We now define Bob's winning strategy. At round $n$, Alice has made a move $A_n$ which has length $(\alpha\beta)^{n-1}\alpha$. This move must contain at least two intervals of length $(\alpha\beta)^n$, which are intervals at the $n$-th stage construction of $P$. Bob picks one of these intervals that does not contain $a_n$. This defines a strategy for Bob such that every run of the strategy will be in $P \setminus S \subseteq S^C$. 

If $A$ is a Luzin set, we construct the same $P$ and simply note that $P$ is meager, so the strategy is defined in exactly the same way. 
\end{proof}

Theorem~\ref{main} shows that for certain choices of $\alpha, \beta$ the Schmidt game for Vitali sets can be determined. In contrast we have the following.

\begin{theorem} For any non-trivial $\alpha$, $\beta$ there is a Vitali set which is not $(\alpha, \beta)$ determined.
\end{theorem}

\begin{proof} By Theorem 2.2 from \cite{M}, there exists a Vitali set that is also a Bernstein set. Being a Bernstein set, by Theorem 3 from \cite{AFS}, this set is never $(\alpha,\beta)$ determined for any non-trivial value of $\alpha$ and $\beta$.
\end{proof}

The figure below illustrates the determinacy of the $(\alpha,\beta)$ Schmidt game on Vitali sets for given values of $\alpha$ and $\beta$. 
Several questions arise as a result of Theorem 2 and Theorem 4.
\begin{center}
\begin{tikzpicture}
    \draw (0, 4) -- (0, 0) -- (4, 0) -- (4, 4) -- (0, 4);
    \draw (4/12, 4) -- (4/12, 1.959184/2);
    \draw (0, 0) parabola (4/12, 1.959184/2);
    \draw (0.1, 0.1) -- (4, 4);
    \node at (1.5, 2.75) {\Large{?}};
    \node at (2.33, 0.53) {\small No $V$ is ($\alpha, \beta$)-winning};
    \node at (1, 4.5) {\small There exists $(\alpha, \beta)-$winning $V$};
    \draw[->] (0.15, 4.2) -- (0.15, 3.9);
    \node at (-0.5, 2) {\Large$\beta$};
    \node at (2, -0.3) {\Large$\alpha$};
    \node at (4.2, -0.3) {(1, 0)};
    \node at (-0.5, 4) {(0, 1)};
    \node at (-0.5, -0.3) {(0, 0)};
\end{tikzpicture} \end{center}

\begin{question}
For $\alpha, \beta$ in the ? area, do there exist Vitali sets that are winning?
\end{question}


\end{document}